\documentclass[12pt]{article}
\usepackage{geometry}                
\usepackage{graphicx}
\usepackage{bm}

\usepackage{amssymb,amsmath}
\usepackage{epstopdf}
\usepackage{amssymb,latexsym,amsmath,amsthm,verbatim}
\usepackage{graphicx,epsfig,epstopdf,amssymb,color}

\newtheorem{theorem}{Theorem}[section]

\newtheorem{proposition}[theorem]{Proposition}
\newtheorem{corollary}[theorem]{Corollary}
\newtheorem{remark}[theorem]{Remark}

\newcommand{\uu}{{\bf u}}
\newcommand{\vv}{{\bf v}}
\newcommand{\xx}{{\bf x}}

\newcommand{\VV}{{\bf V}}

\newcommand{\cE}{{\cal E}}
\newcommand{\cL}{{\cal L}}

\newcommand{\real}{\mathbb{R}}

\def\QED{\mbox{}\hfill$\Box$}

\newcommand{\inttwo}{{\int_{\real^2}}}
\title{Generalized Helmholtz-Kirchhoff model for two dimensional
distributed vortex motion}
\author{

Raymond Nagem$^{1}$ \and Guido Sandri$^{1}$
 \and David Uminsky$^{2,3}$\and C. Eugene Wayne$^{2,3}$}

\begin{document}
\maketitle
\footnotetext[1]{Department of Aerospace and Mechanical Engineering, Boston University}
\footnotetext[2]{Dept. of Mathematics and Statistics and Center for Biodynamics,
 Boston
University}
\addtocounter{footnote}{2}
\footnotetext[3]{Supported in part by NSF grant DMS-0405724}
\addtocounter{footnote}{3}

{
\begin{abstract}
The two-dimensional Navier-Stokes equations are rewritten as a system
of coupled nonlinear ordinary differential equations.  These equations
describe the evolution of the moments of an expansion of the vorticity
with respect to Hermite functions and of the centers of
vorticity concentrations.  We prove the convergence
of this expansion and show that in the zero viscosity and zero core
size limit we
formally recover the Helmholtz-Kirchhoff model for the evolution of
point-vortices.  The present expansion systematically
incorporates the effects of both viscosity and finite vortex
core size.  We also show that a low-order truncation of our
expansion leads to the representation of the flow as a system of
interacting Gaussian (i.e. Oseen) vortices which previous experimental
work has shown to be an accurate approximation to many important
physical flows \cite{Meunier:2005}.
\end{abstract}
}

\section{Introduction}\label{sec:one}

{
In this paper we represent solutions of the two-dimensional 
Navier-Stokes equations as a system of interacting vortices.
This expansion, which generalizes the Helmholtz-Kirchhoff model
of interacting point vortices in an inviscid fluid, systematically
incorporates the effects of both vorticity and finite vortex
core size.  Furthermore, we give conditions which guarantee
the convergence of our expansion.
}
Incompressible viscous flow has two standard analytic
representations: a formulation in terms of the primitive velocity 
and pressure variables, and a formulation in terms of the velocity 
and vorticity variables \cite{majda:2002}.  The velocity-vorticity representation has
particular advantages when boundaries are unimportant,
since vorticity cannot be created or destroyed in the interior
of a fluid.  The vorticity field can also be directly related
to physically observed flow structures such as line and 
ring vortices.
   
In two space dimensions, the vorticity field has the additional
advantage of reducing to a scalar.  An early representation
of two-dimensional flow in terms of moving point vortices
was developed by Helmholtz-Kirchhoff \cite{Kirchhoff:1876} { and by Helmholtz \cite{Helmholtz:1858} }. {The point 
vortex model has been studied extensively - a thorough review of the model and} recent developments are 
described in the monograph by Newton \cite{Newton:2001}.  While the 
Helmholtz-Kirchhoff point vortex model captures many of the basic physical 
phenomena observed in two-dimensional rotational flows, experiments 
with even relatively simple vortex configurations exhibit complications 
far beyond the point vortex predictions \cite{Meunier:2005}. 
{Additionally, the classical point vortex model neglects the effects of {\it viscosity}.}  
{
However, these experiments also reinforce the idea that in many
circumstances the fluid flow may be well approximated by a collection of interacting
vortices - albeit vortices  with finite core size,
subject to the effects of viscosity. { A few recent studies of the interaction of viscous vortices can be found in \cite{CerretelliWilliamson:2003, Dizes:1999,DizeVerga:2002,Melander:2006,Meunier:2005,ting:1991}. The main focus of these papers is the merger of two like signed vortices. In \cite{Melander:2006} the authors use a spatial moment model for 2-D Euler equations which later incorporates weak Newtonian viscosity to derive equations of motions for two like signed vortices. A metastable state is found before merger which consists of two rotating, near-circular, vortices.  More recently, it has been conjectured that in a two-vortex system the profiles relax to a pair of gaussian vortices before merging \cite{DizeVerga:2002} .    Thus, it is of interest to extend and generalize the 
Helmholtz-Kirchhoff point vortex model to a model that incorporates non-zero vortex core size and viscous effects while retaining its basic form.}  Such an extension is the goal of this paper.
}

The governing equations for the velocity $({\bf u})$ and 
pressure $(p)$ variables are
\begin{equation}\label{fluid1}
\frac{\partial {\bf u}}{\partial t} + ({\bf u} \cdot \nabla){\bf u} = 
- \nabla \frac{p}{\rho} + \nu \Delta {\bf u},
\end{equation}
\begin{equation}\label{fluid2}
\nabla \cdot {\bf u} = 0,
\end{equation}
where $\rho$ is the fluid density and $\nu$ is the kinematic
viscosity.  Taking the curl of (\ref{fluid1}) and (\ref{fluid2}) 
gives
\begin{equation}\label{fluids4}
\frac{\partial {\bm \omega}}{\partial t} +
({\bf u} \cdot \nabla) {\bm \omega} -
({\bm \omega} \cdot \nabla) {\bf u} = \nu \Delta {\bm \omega},
\end{equation}
\begin{equation}\label{fluids5}
{\bm \omega} = \nabla \times {\bf u}, \; \; \;
\nabla \cdot {\bm \omega} = 0,
\end{equation}
which are the governing equations for the velocity-vorticity
variables.  For two-dimensional flows, the vorticity vector 
${\bm \omega}$ is perpendicular to the plane of the flow, and
the third term on the left-hand side of (\ref{fluids4}) vanishes.
The condition $\nabla \cdot {\bm \omega} = 0$ is identically
satisfied, and equation (\ref{fluids4}) then reduces to the
single scalar equation:
\begin{equation}\label{fluids6}
\frac{\partial  \omega}{\partial t} +
({\bf u} \cdot \nabla) \omega = \nu \Delta \omega,
\end{equation}
where $\omega$ is the single, non-zero component of the vorticity.
A drawback of the formulation in \eqref{fluids6} is that the velocity
of the fluid is still present in the equation.
However, assuming that the vorticity field is sufficiently localized,
the velocity vector can be computed in terms of the 
vorticity $\omega$
by the Biot-Savart law
\begin{equation}\label{fluids7}
{\bf u}(x) = \frac{1}{2\pi} \inttwo \frac{(x-y)^{\perp}}{|x-y|^2} \omega(y) dy\ ,
\end{equation}
where for a two-vector $z = (z_1,z_2)$, $z^{\perp} = (-z_2,z_1)$.

In this paper we use equations (\ref{fluids6}) and 
(\ref{fluids7}) to develop a vorticity representation of 
two-dimensional viscous flow.   Our representation 
is based on a decomposition of the vorticity field into
a set of moving distributed vortices.  Differential equations 
are derived for the motion of the vortex centers and for
the time evolution of the vortex distributions.  
{
The evolution of each individual vortex is represented as an expansion with respect to
a sequence of Hermite functions.  Such expansions have proven useful in theoretical
studies of two-dimensional fluid flows (\cite{GallayWayne:2002}, \cite{GallayWayne:2005}) and 
the leading order term in this expansion is precisely the Gaussian vortex 
(i.e. Oseen vortex \cite{oseen:1927},\cite{oseen:1911}) whose
utility as an approximation for vortex interaction was shown in \cite{Meunier:2005}.
We show that the coefficients in this expansion satisfy a system of ordinary differential
equations whose coefficients can be explicitly represented in terms of a fixed,
computable kernel function.  We also prove the convergence of this expansion.
}
It is shown that our representation reduces in the appropriate 
limit to the Helmholtz-Kirchhoff model, and allows at the same time
arbitrarily complex evolution and interaction of the 
moving vortices.   

In the present paper we concentrate on the mathematical formulation of 
the generalized Helmholtz-Kirchhoff model.  In future work we will explore the
predictions of this model both numerically and analytically in a number of
different physical settings.

\section{The ``multi-vortex'' expansion}\label{sec:two}

In this section we separate the solution of the vorticity equation into $N$  components and derive separate evolution equations for
each component.  From a physical point of view this decomposition will be most useful 
when each of the  components corresponds
to a localized region of vorticity (e.g. a vortex) well separated from the other lumps but the mathematical development
described below is well defined {without} regard to these physical considerations.  However,
with this application in mind, we will often refer to each of the components as a
``vortex''.

Consider the initial value problem for the two-dimensional vorticity equation
\begin{eqnarray}\label{eq:2DIVP}\nonumber
&& \frac{\partial \omega}{\partial t} = \nu \Delta \omega - \uu \cdot \nabla \omega ,\\
&& \omega = \omega(x,t) ,\ x \in \real^2 ,\ t  > 0 \\ \nonumber
&& \omega(x,0) = \omega_0(x)
\end{eqnarray}
where $\uu$ is the velocity field associated to the vorticity field $\omega$.  We begin by decomposing the initial vorticity
distribution by writing
\begin{equation}\label{eq:initial_decomposition}
\omega_0(x) = \sum_{j=1}^N \omega^j_0(x)\ .
\end{equation}
Of course this decomposition is not unique - even the number of pieces, $N$, into which we decompose the
vorticity is up to us to choose.  In general the choice we make will be motivated by physical considerations,
however,  for the development below,
all we require of the decomposition is that the total vorticity of each vortex is non-zero, i.e.
\begin{equation}\label{eq:mass}
m_j = \int_{\real^2}  \omega^j_0(x) dx \ne 0 ,\ j=1, \dots ,N .
\end{equation}
If \eqref{eq:mass} is satisfied we define $x^j_0$ by
\begin{equation}\label{eq:COMIC}
\int_{\real^2} (x-x^j_0) \omega^j_0(x) dx = 0\ ,
\end{equation}
or equivalently
\begin{equation}\label{eq:COMIC2}
x^j_0 = \frac{1}{m_j} \inttwo x \omega^j_0(x) dx\ .
\end{equation}

We now write the vorticity for $t>0$ as
\begin{equation}\label{eq:vorticity_decomposition}
\omega(x,t) = \sum_{j=1}^N \omega^j(x-x^j(t);t)
\end{equation}
and the velocity field as
\begin{equation}\label{eq:velocity_decomposition}
\uu(x,t) = \sum_{j=1}^N \uu^j(x-x^j(t);t)
\end{equation}
where $\uu^j(y,t)$ is the velocity field associated to $\omega^j(y,t)$ by the Biot-Savart Law.
Of course we still have to define the equations of motion for $\omega^j(y,t)$ and $x^j(t)$.

The centers of the vorticity  regions, $x^j(t)$, and the vorticity regions themselves
evolve via a coupled system of ordinary-partial differential equations constructed
so that in the limit of zero viscosity and when the different components of the
vorticity happen to
be point vortices (i.e. Dirac-delta functions) we recover the {Helmholtz-Kirchhoff} point
vortex equations.  If we take the partial derivative of \eqref{eq:vorticity_decomposition}
and use the equation satisfied by the vorticity, we find:
\begin{eqnarray}\label{eq:allomega}
\partial_t \omega(x,t)&=& \sum_{j=1}^N \partial_t \omega^j(s-s^j(t),t) -
\sum_{j=1}^N \dot{x}^j(t) \cdot \nabla \omega^j(x-x^j(t),t) \\ \nonumber
&=& \sum_{j=1}^N \nu \Delta \omega^j(x-x^j(t),t) - \sum_{j=1}^N \left(\sum_{\ell=1}^N 
\uu^{\ell}(x-x^{\ell}(t),t) \right) \cdot \nabla \omega^j(x-x^j(t),t)\ .
\end{eqnarray}

{Given this equation  it is natural to define $\omega^j$ as the solution of  the equation}:
\begin{eqnarray}\label{eq:omegaj}\nonumber
\frac{\partial \omega^j}{\partial t}(x-x^j(t),t) &=& \nu \Delta \omega^j(x-x^j(t),t) 
-  \left(\sum_{\ell=1}^N 
\uu^{\ell}(x-x^{\ell}(t),t) \right) \cdot \nabla \omega^j(x-x^j(t),t) \\ 
&& \qquad +  \dot{x}^j(t) \cdot \nabla \omega^j(x-x^j(t),t) 
\ ,\ j=1, \dots , N.
\end{eqnarray}

To close this system of equations we must specify how the centers of vorticity
$x^j(t)$ evolve.  We impose the condition that the first moment of each
vorticity region must vanish at every time $t > 0$, {i.e.} we require that
\begin{equation}\label{eq:xjdef}
\int_{\real^2}(x-x^j(t))  \omega^j(x-x^j(t),t) dx = 0 ~~{\rm for~all} ~~t>0, \  j = 1, \dots N\ .
\end{equation}
(Note that this equation really contains two conditions - one for each component
of $(x-x^{\ell}(t))$.)
{We impose this condition to fix  the evolution of $x^j(t) $ because Gallay and Wayne have recently shown \cite{GallayWayne:2002} that if one considers the
evolution of general solutions of \eqref{eq:2DIVP}  the solution will approach an 
Oseen vortex, and the rate of the approach will be faster if the vorticity distribution
has first moment equal to zero.}   Solutions of \eqref{eq:2DIVP} preserve the
first moment, and hence if the initial conditions have first moment equal
to zero the solution will have first moment zero for all time.  The equations
\eqref{eq:omegaj} no longer preserve the first moment and thus we impose
this condition for all time, which then defines the motion of the center of vorticity.

Note that if we change variables in \eqref{eq:xjdef} to $z=x-x^j(t)$, we find
\begin{equation}
\int_{\real^2} z \omega^j(z,t) dz = 0\ .
\end{equation}
Since this equation holds for all $t > 0$ we can differentiate both sides with
respect to $t$ to obtain
\begin{equation}
\int_{\real^2} z \partial_t  \omega^j(z,t) dz = 0\ .
\end{equation}

Using \eqref{eq:omegaj} we can insert the formula for $\partial_t \omega^j$ into
this integral and we obtain:
\begin{eqnarray}\label{eq:xjdotintermediate} \nonumber
&& \nu \inttwo z \Delta \omega^j(z,t) dz - \inttwo z \left(\sum_{\ell=1}^N 
\uu^{\ell}(z + x^j(t) -x^{\ell}(t),t) \right) \cdot \nabla \omega^j(z,t) dz \\
&& \qquad \qquad + \inttwo z \left(  \dot{x}^j(t) \cdot \nabla \omega^j(z,t) \right) dz = 0 \ . 
\end{eqnarray}
We first note that if we integrate twice by parts, we have
\begin{equation}\label{eq:ibp}
\inttwo z \Delta \omega^j(z,t) dz = 0\ .
\end{equation}
Next if we take the $n^{{\rm th}}$ component and integrate by parts we find
\begin{equation}\label{eq:ibpagain}
\inttwo z \left(  \dot{x}^j_n (t) \cdot \nabla \omega^j(z,t) \right) dz
= \dot{x}^j_n(t) \inttwo z \partial_{z_n} \omega^j(z,t) dz =  - m_j \dot{x}^j_n(t)\ ,
\end{equation}
where $m_j = \inttwo \omega^j(z,t) dz$ and $n=1,2$.
\begin{remark} Note that the equations \eqref{eq:omegaj} do preserve the total
integral (``mass'') of the solution so this definition of $m_j$ is consistent with
\eqref{eq:mass}.
\end{remark}

Finally, recalling that the velocity field is incompressible we can rewrite
$$
\left(\sum_{\ell=1}^N 
\uu^{\ell}(z + x^j(t) -x^{\ell}(t),t) \right) \cdot \nabla \omega^j(z,t) 
= \nabla \cdot \left(\sum_{\ell=1}^N 
\uu^{\ell}(z + x^j(t) -x^{\ell}(t),t)  \omega^j(z,t) \right)
$$
so, again considering the $n^{{\rm th}}$ component and integrating by parts we have
\begin{eqnarray}\label{eq:ibponcemore}
&& \inttwo z_n \left(\sum_{\ell=1}^N 
\uu^{\ell}(z + x^j(t) -x^{\ell}(t),t) \right) \cdot \nabla \omega^j(z,t) dz \\ \nonumber 
&& \qquad \qquad \qquad  =
- \inttwo  \left(\sum_{\ell=1}^N 
\uu^{\ell}_n (z + x^j(t) -x^{\ell}(t),t) \right) \omega^j(z,t) dz
\end{eqnarray}

Thus, if we combine \eqref{eq:ibp} \eqref{eq:ibpagain}, and \eqref{eq:ibponcemore}
we see that \eqref{eq:xjdotintermediate} reduces to the system of ordinary differential
equations for the centers of the vorticity distributions:
\begin{equation}\label{eq:xjdot}
\frac{d x^j_n}{dt}(t) = \frac{1}{m_j} \sum_{\ell=1}^N \inttwo  \left(
\uu^{\ell}_n (z + x^j(t) -x^{\ell}(t),t) \omega^j(z,t)  \right) dz\ ,
\end{equation}
supplemented by the initial conditions \eqref{eq:COMIC}, while the
$N$ components of the vorticity evolve according to the partial differential equations \eqref{eq:omegaj}
with initial conditions 
\begin{equation}\label{eq:omegajIC}
\omega^j(z,0) = \omega^j_0(z+x^j_0)
\end{equation}
obtained by combining \eqref{eq:initial_decomposition} and \eqref{eq:vorticity_decomposition}.

\begin{remark} Consider \eqref{eq:xjdot} in the  limit in which the components $\omega^j$
are all point vortices, i.e. $\omega^j(z,t) = m_j \delta(z)$, with $\delta(z)$ the {Dirac-delta} function.  Recall that the velocity field associated with such a point vortex is
$$
U_n(z_1,z_2,t) = -\sum_{j=1}^2 \epsilon_{n,j} z_j \frac{1}{(z_1^2+z_2^2)}
$$
where $\epsilon_{m,j}$ is the antisymmetric tensor with two indices.
Then if we ignore the (singular) term with $\ell =j$ in the sum on the right hand
side of \eqref{eq:xjdot} we find 
\begin{equation}\label{eq:Helmholtz}
\frac{d x^j_n}{dt}(t) =  \sum_{\ell=1; \ell\ne j}^N m_{\ell} \frac{\left( \sum_{k=1}^2 \epsilon_{n,k}(x^j(t) -x^{\ell}(t))_k 
\right) }{|x^j(t) -x^{\ell}(t)|^2}
\end{equation}
These of course are just {Helmholtz-Kirchhoff} equations for the inviscid motion of a system of 
point vortices.  Thus, our expansion can be regarded as a generalization of 
the this
approximation which allows for both nonzero viscosity and vortices of finite size.
To justify omitting the term with $\ell = j$ on the
right hand side of \eqref{eq:Helmholtz} we note that if we approximate the delta
function with a narrow, Gaussian vorticity distribution, and $\uu^j$ by the 
corresponding velocity field, this term will vanish by symmetry.
\end{remark}

\section{The moment expansion; case of a single center}\label{sec:three}

In this section we introduce another idea - an expansion of the vorticity in
terms of Hermite functions.  Then, in the next section we will combine the 
Hermite expansion with the multi-vortex expansion of the previous section.

The moment expansion is an expansion of the solution of the vorticity
equation in terms of Hermite functions.  Define
\begin{equation}\label{eq:phi00}
\phi_{00}(x,t;\lambda) = \frac{1}{ \pi \lambda^2 } e^{-|x|^2/\lambda^2}
\end{equation}
where $\lambda^2 = \lambda_0^2 + 4 \nu t$.  Three simple facts that we will use
repeatedly are
\begin{enumerate}
\item[(i)] $ \partial_t \phi_{00} = \nu \Delta \phi_{00}$
\item[(ii)] $\inttwo \phi_{00}(x,t;\lambda) dx = 1\ \ {\rm for ~ all}\ \ t \ge 0\ .$
\item[(iii)] Finally, and crucially for what follows, the vorticity function
$\omega(x,t) = \alpha \phi_{00}(x,t)$ is an exact solution (called the Oseen,
or Lamb, vortex) of the two dimensional vorticity equation for all values of $\alpha$.
\end{enumerate}
Note that we will often supress the dependence of $\phi_{00}$ on $\lambda$ when
there is no fear of confusion.

We now define the Hermite functions of order $(k_1,k_2)$ by
\begin{equation}\label{eq:Hermitedef}
\phi_{k_1,k_2}(x,t;\lambda) = D_{x_1}^{k_1} D_{x_2}^{k_2} \phi_{00}(x,t;\lambda)
\end{equation}
and the corresponding moment expansion of a function by
\begin{equation}\label{eq:moment_expansion_def}
\omega(x,t) = \sum_{k_1, k_2=1}^{\infty} M[k_1,k_2;t] \phi_{k_1,k_2}(x,t;\lambda) 
\end{equation}

Note that if the function $\omega(x,t)$ in \eqref{eq:moment_expansion_def} is the
vorticity field of some fluid the linearity of the Biot-Savart law implies that we
can expand the associated velocity field as:
\begin{equation}\label{eq:velocity_field_expansion}
\VV(x,t) =  \sum_{k_1, k_2=1}^{\infty} M[k_1,k_2;t] \VV_{k_1,k_2}(x,t;\lambda) 
\end{equation}
where
\begin{equation}\label{eq:velocity_moments}
\VV_{k_1,k_2}(x,t;\lambda) = D_{x_1}^{k_1} D_{x_2}^{k_2} \VV_{00}(x,t;\lambda)
\end{equation}
and $\VV_{00}(x,t;\lambda)$ is the velocity field associated with the Gaussian
vorticity distribution $\phi_{00}$ - explicitly we have:
\begin{equation}\label{eq:Vzerodef}
\VV_{00}(x,t;\lambda) =  \frac{1}{2\pi}\frac{(-x_2,x_1)}{|x|^2}
  \Bigl(1 -  e^{-|x|^2/\lambda^2}\Bigr)~,
\end{equation}

We define the Hermite polynomials via their generating function:
\begin{equation}\label{eq:Hermitepoly}
H_{n_1,n_2}(z;\lambda) =  \left( (D^{n_1}_{t_1} D^{n_2}_{t_2}
e^{\left( \frac{2 t\cdot z -  t^2}{\lambda^2}\right)} \right)|_{t=0}
\end{equation}
Note that the ``standard'' Hermite polynomials correspond to taking $\lambda=1$.

Then using the standard orthogonality relationship for the Hermite polynomials:
\begin{equation}
\inttwo H_{n_1,n_2}(z;\lambda=1) H_{m_1,m_2}(z;\lambda=1) e^{-z^2} dz
= \pi 2^{n_1+n_2} (n_1!) (n_2!) \delta_{n1,m1} \delta_{n_2,m_2}
\end{equation}
we see that the coefficients in the expansion \eqref{eq:moment_expansion_def} are
defined by the projection operators:
\begin{equation}\label{eq:coefficient}
 M[k_1,k_2;t]  = (P_{k_1,k_2} \omega)(t) = \frac{ (-1)^{k_1+k_2} \lambda^{2(k_1+k_2)} }{2^{k_1+k_2} (k_1!) (k_2!) }
 \inttwo H_{k_1,k_2}(z;\lambda) \omega(z,t) dz
\end{equation}

\subsection{Convergence of the moment expansion}\label{subsec:onevortexconvergence}

In this subsection we derive a criterion for the convergence of the moment expansion
derived above and we show that if this criterion is satisfied for $t=0$ then it is satisfied
for all subsequent times $t>0$.

Our convergence criterion is based on the observation that the Hermite
functions $\phi_{k_1,k_2}(x,t;\lambda) $ are, for any value of $t$, the eigenfunctions
of the linear operator
\begin{equation}\label{eq:Llambda_def}
\cL^{\lambda} \psi = \frac{1}{4} \lambda^2 \Delta \psi + \frac{1}{2} \nabla \cdot (x \psi)\ .
\end{equation}
This fact can be verified by direct computation and is related to the fact that
$\cL^{\lambda}$ can be transformed into the Hamiltonian quantum mechanical harmonic
oscillator.

The Gaussian function $\phi_{00}$ plays a crucial role in the convergence
proof, and its dependence {on} the parameter $\lambda$ is particularly important in this
discussion, so for this subsection only we will define
$$
\phi_{00}(x,t;\lambda) = \Phi_{\lambda}(x,t)
$$
to empasize this dependence.

If one now proceeds as in Lemma 4.7 of \cite{GallayWayne:2005} one can prove
\begin{proposition}\label{prop:selfadjoint} 
The operator $\cL^{\lambda}$ is self-adjoint in the Hilbert space
$$
X^{\lambda} = \{ f \in L^2(\real^2) ~|~ \Phi_{\lambda}^{-1/2} f \in L^2(\real^2) \}
$$
with innerproduct $(f,g)_{\lambda} = \int_{\real^2} \Phi_{\lambda}^{-1} \overline{f} g dx$.
\end{proposition}

An immediate corollary of this proposition and the general theory of self-adjoint operators is
\begin{corollary}\label{corr:self_adjoint}
The eigenfunctions of $\cL^{\lambda}$ form a complete orthogonal set 
in the Hilbert space $X^{\lambda}$.
\end{corollary}
\noindent
and as a corollary of this result  and the observation that the eigenfunctions
of $\cL^{\lambda}$ are precisely our Hermite functions $\{ \phi_{k_1,k_2} \}$,
we have finally
\begin{proposition}\label{prop:convergence_criterion}
Suppose that {
$$
\| f \|_{\lambda}^2 = \int_{\real^2} \Phi_{\lambda}^{-1}(x) |f(x) |^2 dx < \infty\ ,
$$
then the expansion}
$$
f(x) = \sum_{k_1,k_2} M[k_1,k_2] \phi_{k_1,k_2}(x)
$$
converges with respect to the norm on the Hilbert space $X^{\lambda}$.
\end{proposition}

Thus, { the following criterion guarantees that the expansion  \eqref{eq:moment_expansion_def} 
for the vorticity converges:}
\begin{equation}\label{eq:convergence_criterion}
\int_{\real^2} \Phi_{\lambda}^{-1}(x) (\omega(x,t))^2 dx  < \infty\ .
\end{equation}
The main result of this subsection is the following theorem which proves that if
our initial vorticity distribution satisfies \eqref{eq:convergence_criterion} for some
$\lambda=\lambda_0$, then the solution of the vorticity equation with that initial
condition will satisfy \eqref{eq:convergence_criterion} for all time $t$ with
$\lambda = \sqrt{4 \nu t + \lambda_0^2}$ and hence as a corollary if the initial vorticity
distribution satisfies \eqref{eq:convergence_criterion} then our moment expansion
converges for all times $t$.
\begin{theorem}\label{th:converence_criterion}  Define the weighted enstrophy function
$$
\cE(t) = \int_{\real^2} \Phi_{\lambda}^{-1}(x) (\omega(x,t))^2 dx \ .
$$
If the initial vorticity distribution $\omega_0$ is such that 
$\cE(0) < \infty$ for some $\lambda_0$, and $\omega_0$ is bounded
(in the $L^{\infty}$ norm) then $\cE(t)$ is finite
for all times $t > 0$.
\end{theorem}
{\bf Proof:}  The idea of the proof is to derive a differential inequality for $\cE(t)$ which
guarantees that if $\cE(0)$ is finite then $\cE(t)$ will be finite for all $t$.
Differentiating $\cE(t)$ we obtain
\begin{eqnarray} \label{eq:dEone}
&& \frac{d \cE}{dt}(t) = \frac{4 \nu}{\lambda^2} \cE(t) - \frac{4 \nu}{\lambda^4}
\inttwo |x|^2  \Phi_{\lambda}^{-1}(x) (\omega(x,t))^2 dx \\ \nonumber && \qquad \qquad \quad  + 
2 \inttwo \Phi_{\lambda}^{-1}(x) \omega(x,t) \partial_t \omega(x,t) dx \\ \label{eq:dEtwo}
 && \qquad
= \frac{4 \nu}{\lambda^2} \cE(t) - \frac{4 \nu}{\lambda^4}
\inttwo |x|^2  \Phi_{\lambda}^{-1}(x) (\omega(x,t))^2 dx \\ \nonumber && \qquad \qquad  \quad + 
2 \inttwo \Phi_{\lambda}^{-1}(x) \omega(x,t) \left( \nu \Delta \omega - \uu \cdot \nabla \omega\right) dx
\end{eqnarray}

We now consider the last term in \eqref{eq:dEtwo}.   First note that upon integration by
parts we have:
\begin{equation}\label{eq:dEthree}
2 \inttwo \Phi_{\lambda}^{-1}(x) \omega(x,t) \left( \nu \Delta \omega(x,t) \right) dx
= -  2 \nu \inttwo \Phi_{\lambda}^{-1}(x) \left( | \nabla \omega|^2 +  \frac{2}{\lambda^2} \omega
x \cdot \nabla \omega \right) dx\ .
\end{equation} 
The right hand side of the expression in \eqref{eq:dEthree} can again be broken
up into two pieces and the second can be bounded by:
\begin{equation}\label{eq:dEfour}
2 \nu \inttwo \Phi_{\lambda}^{-1}(x)\left( \frac{2}{\lambda^2} \omega
x \cdot \nabla \omega \right) dx \le  \nu \inttwo \Phi_{\lambda}^{-1}(x) |\nabla \omega |^2 dx +
\frac{4\nu }{\lambda^4}  \inttwo \Phi_{\lambda}^{-1}\left (x^2 \omega^2 \right) dx \ .
\end{equation}

{ Finally,} we bound the last term in \eqref{eq:dEone}, which comes from the nonlinear
term in the vorticity equation.  In this estimate we use the fact 
(see \cite{GallayWayne:2002}, Lemma 2.1) that
the $L^{\infty}$ norm of the velocity field $\uu$ can be bounded by a constant times
the sum
of the $L^1$ and $L^{\infty}$ norms of the vorticity field - i.e by
$C(\| \omega \|_{L^1(\real^2)} + \| \omega \|_{L^{\infty}(\real^2)} )$.  This observation,
combined with the fact that $\| \omega(\cdot,t) \|_{L^p(\real^2)} \le 
\| \omega_0 \|_{L^p(\real^2)}$, which is a consequence of the maximum principle,
implies that 
$$
\| \uu(\cdot,t) \|_{L^{\infty}(\real^2)} \le 
C (\| \omega_0 \|_{L^1(\real^2)} + \| \omega_0 \|_{L^{\infty}(\real^2)} )\ .
$$
and hence that 
\begin{eqnarray}\label{eq:dEfive}
&& 2 \inttwo \Phi_{\lambda}^{-1}(x) \omega(x,t) \left( \uu \cdot \nabla \omega\right) dx
\le 2 C(\omega_0) \inttwo \Phi_{\lambda}^{-1}(x) |\omega(x,t) | |  \nabla \omega |dx \\ \nonumber
&&\qquad \qquad \le  \frac{4 C(\omega_0)}{\nu}  \inttwo \Phi_{\lambda}^{-1}(x) (\omega(x,t) )^2 dx
+  \nu \inttwo \Phi_{\lambda}^{-1}(x) |  \nabla \omega |^2 dx 
\end{eqnarray}

If we now combine the inequalities in \eqref{eq:dEthree}, \eqref{eq:dEfour} and
\eqref{eq:dEfive}, with the expression for $\frac{d\cE}{dt}$ in \eqref{eq:dEtwo} we 
obtain:
\begin{equation}
\frac{d\cE}{dt}(t) \le  \left( \frac{4 C(\omega_0)}{\nu} +  \frac{4 \nu}{\lambda^2} \right) \cE(t)\ ,
\end{equation}
from which we see immediately that if $\cE(0)$ is bounded, $\cE(t)$ remains bounded for
all time.
\QED

\subsection{Differential equations for the moments}

Assuming that the function $\omega(z,t)$ is a solution of \eqref{eq:2DIVP},  we
can derive differential equations satisfied by the moments 
$M[i_1,k_2,t]$ in \eqref{eq:moment_expansion_def}.  Surprisingly the expressions for the
coefficients in these expansions are quite simple and explicit.

If we differentiate  \eqref{eq:moment_expansion_def} and assume that 
$\omega$ is a solution of the two-dimensional vorticity equation { we obtain}:
\begin{eqnarray}\nonumber
&& \partial_t \omega = \sum_{k_1, k_2=1}^{\infty} \frac{dM[k_1,k_2;t] }{dt}
\phi_{k_1,k_2}(x,t;\lambda)  + \sum_{k_1, k_2=1}^{\infty} 
M[k_1,k_2;t] \partial_t \phi_{k_1,k_2}(x,t;\lambda)  \\ && \quad
= \sum_{k_1, k_2=1}^{\infty} M[k_1,k_2;t] \left( \nu \Delta \phi_{k_1,k_2}(x,t;\lambda)  \right) \\ \nonumber
&& 
\qquad 
\qquad
- \left( \sum_{\ell_1, \ell_2=1}^{\infty} M[\ell_1,\ell_2;t] \VV_{\ell_1,\ell_2}(x,t;\lambda) \right)
\cdot \nabla \left(\sum_{k_1, k_2=1}^{\infty} M[k_1,k_2;t] \phi_{k_1,k_2}(x,t;\lambda)  \right)
\end{eqnarray}

From the first of the ``simple facts'' we stated about $\phi_{00}$,
we see that the last term on the first line cancels the
middle line and hence if we apply the projection operators
defined in \eqref{eq:coefficient},  we are left with the 
system of ordinary differential equations for the moments
\begin{eqnarray}\label{eq:mom1}
\frac{dM[k_1,k_2;t] }{dt} &=&  - P_{k_1,k_2} \left[  \left( \sum_{\ell_1, \ell_2=1}^{\infty} M[\ell_1,\ell_2;t] 
\VV_{\ell_1,\ell_2}(x,t;\lambda) \right)\right. \\ \nonumber
&& \qquad  \qquad 
\cdot \nabla\left.  \left(\sum_{m_1, m_2=1}^{\infty} M[m_1,m_2;t] \phi_{m_1,m_2}(x,t;\lambda)  \right) \right]
\end{eqnarray}

The somewhat surprising fact {which, in our opinion, makes the preceding} straightforward
calculations interesting is that the projection on the right hand side of \eqref{eq:mom1}
can be computed explicitly in terms of the derivatives of an relatively simple function.

We now explain how this is done.  First recall that:
\begin{equation}\label{eq:derivatives}
\phi_{m_1,m_2}(x,t;\lambda) = D_{x_1}^{m_1} D_{x_2}^{m_2} \phi_{00}(x,\lambda),\ 
\VV_{\ell_1,\ell_2}(x,t;\lambda) = D_{x_1}^{\ell_1} D_{x_2}^{\ell_2} \VV_{00}(x,\lambda)
\end{equation}
In order to avoid confusing the two sets of derivatives we will rewrite these formulas
as
\begin{eqnarray}\label{eq:derivativestoo}
\phi_{m_1,m_2}(x,t;\lambda) &=&  (D_{a_1}^{m_1} D_{a_2}^{m_2} \phi_{00}(x+a,\lambda))|_{a=0},\\
\nonumber
\VV_{\ell_1,\ell_2}(x,t;\lambda)  &=&  (D_{b_1}^{\ell_1} D_{b_2}^{\ell_2} \VV_{00}(x+b,\lambda))|_{b=0}
\end{eqnarray}
Inserting these formulas into the right hand side of \eqref{eq:mom1} and using the formula
for the projection operator $P_{k_1,k_2}$ in terms of the integration against
a Hermite polynomial we { obtain}
\begin{eqnarray} \label{eq:moment_evolution}
 && \frac{d M}{dt}[k_1,k_2,t] = - \frac{(-1)^{(k_1+k_2)} \lambda^{2 (k_1+ k_2)}}{2^{k_1+k_2} (k_1 !) (k_2 !) } \sum_{\ell_1,\ell_2} \sum_{m_1,m_2} 
  M[\ell_1,\ell_2,t]  M[m_1,m_2,t] \\ \nonumber && \qquad \qquad \qquad \times  \int_{\real^2}  H_{k_1,k_2}(\xx) 
 ( D_{x_1}^{m_1} D_{x_2}^{m_2} \VV_{00}(\xx;\lambda) ) \cdot \nabla_{\xx} ( D_{x_1}^{\ell_1} D_{x_2}^{\ell_2} \phi_{00}(\xx;\lambda) ) d\xx \\
 \nonumber && \quad = - \frac{ (-1)^{(k_1+k_2)} \lambda^{2 (k_1+ k_2)}}{2^{k_1+k_2} (k_1 !) (k_2 !) } \sum_{\ell_1,\ell_2} \sum_{m_1,m_2} 
  M[\ell_1,\ell_2,t]  M[m_1,m_2,t] \\ \nonumber && \quad \times
  D_{t_1}^{k_1} D_{t_2}^{k_2} D_{b_1}^{m_1} D_{b_2}^{m_2} D_{a_1}^{\ell_1} D_{a_2}^{\ell_2} \nabla_{a} \cdot
\left(   \int_{\real^2} e^{\frac{ (-t_1^2-t_2^2 + 2 t_1 x_1 + 2t_2 x_2)}{\lambda^2}} \VV_{00}(x+b;\lambda) 
\phi_{00}(x+a;\lambda) d\xx \right) |_{t=0,a=0,b=0} \ .
\end{eqnarray}
The last equality in this expression {results from rewriting the} Hermite polynomial
$H_{l_1,k_2}$ in terms of its generating function.

The last step in deriving the equations for the moments is to evaluate
the integral in the last line of \eqref{eq:moment_evolution}.  The key step in this
evaluation is to recall that for these incompressible flows the velocity field can be written
in terms of the derivatives of the stream function and that the Laplacian of the 
stream function is minus the vorticity.  Thus, we can write:
\begin{equation}\label{eq:stream}
\VV_{00}(x+b;\lambda) = - \nabla_{b}^* (\Delta_b)^{-1} \phi_{00}(x+b)\ ,
\end{equation}
where $\nabla_{b}^* f = (\partial_{x_2} f,-\partial_{x_1} f)$.

Inserting this into the integral in \eqref{eq:moment_evolution} we find
\begin{eqnarray}\label{eq:integral_kernel}
 &&  \int_{\real^2} e^{\frac{ (-t_1^2-t_2^2 + 2 t_1 x_1 + 2t_2 x_2)}{\lambda^2}} \VV_{00}(x+b;\lambda) 
\phi_{00}(x+a;\lambda) d\xx \\ \nonumber && \qquad \qquad
= - \nabla_b^* (\Delta_b)^{-1}   \int_{\real^2} e^{\frac{ (-t_1^2-t_2^2 + 2 t_1 x_1 + 2t_2 x_2)}{\lambda^2}} \phi_{00}(x+b;\lambda) 
\phi_{00}(x+a;\lambda) d\xx
\end{eqnarray}
Now note that all three factors in the integrand are Gaussians and thus the integral can be
evaluated { explicitly,} and we find
\begin{eqnarray}
&&  \int_{\real^2} e^{\frac{ (-t_1^2-t_2^2 + 2 t_1 x_1 + 2t_2 x_2)}{\lambda^2}} \phi_{00}(x+b;\lambda) 
\phi_{00}(x+a;\lambda) d\xx \\ \nonumber && \qquad \qquad
= \frac{1}{2 \pi \lambda^2} e^{-\frac{1}{2 \lambda^2}\left( a_1^2+a_2^2-2 a_1b_1+b_1^2-2 a_2b_2+b_2^2 + 2 a_1 t_1 + 2 b_1 t_1 + t_1^2  + 2 a_2 t_2 + 2 b_2 t_2 +  t_2^2 \right)}
\end{eqnarray}

We next compute the expression
\begin{eqnarray}\label{eq:kernel_velocity}
&&- \nabla_b^* (\Delta_b)^{-1}  \frac{1}{2 \pi \lambda^2} 
e^{-\frac{1}{2 \lambda^2}\left( a_1^2+a_2^2-2 a_1b_1+b_1^2-2 a_2b_2+b_2^2 
+ 2 a_1 t_1 + 2 b_1 t_1 + t_1^2  + 2 a_2 t_2 + 2 b_2 t_2 +  t_2^2 \right)} \\ \nonumber && \qquad \quad
\end{eqnarray}
Recall that given a vorticity field $\omega$, $-(\Delta)^{-1} \omega$ is the associated
stream function and $-\nabla (\Delta)^{-1} \omega$ the velocity field associated with
$\omega$.  
Since the inverse Laplacian and derivatives in \eqref{eq:kernel_velocity} act only on the
$b$-dependent parts of the expression we need to evaluate
\begin{eqnarray}
&& - \nabla_b^* (\Delta_b)^{-1}  \frac{1}{2 \pi \lambda^2} 
e^{-\frac{1}{2 \lambda^2}\left( -2 a_1b_1+b_1^2-2 a_2b_2+b_2^2 
+  2 b_1 t_1 + 2 b_2 t_2  \right)} \\ \nonumber && \qquad =
- e^{ \frac{1}{2 \lambda^2}\left( (t_1-a_1)^2+(t_2-a_2)^2 \right)} 
\nabla_b^* (\Delta_b)^{-1}  \frac{1}{2 \pi \lambda^2} 
e^{-\frac{1}{2 \lambda^2}\left(  (b_1 + (t_1-a_1))^2 + (b_2 + (t_2-a_2) )^2 \right)}
\end{eqnarray}
But
$$
- \nabla_b^* (\Delta_b)^{-1}  \frac{1}{ \pi \lambda^2} 
e^{-\frac{1}{2 \lambda^2}\left(  (b_1 + (t_1-a_1))^2 + (b_2 + (t_2-a_2) )^2 \right)}
$$
is just the velocity field associated with a Gaussian vorticity distribution (i.e.
an Oseen vortex) centered at the point $- ( (t_1-a_1),(t_2-a_2) )$ which we know explicitly.
Hence, the expression on the right hand side of \eqref{eq:integral_kernel} has the explicit
representation:
\begin{eqnarray}\label{eq:integral_kernel_redux}
&& \frac{1}{2 \pi } e^{-\frac{1}{2 \lambda^2}\left( a_1^2 +a_2^2 +t_1^2+t_2^2 \right)}
 e^{ \frac{1}{2 \lambda^2}\left( (t_1-a_1)^2+(t_2-a_2)^2 \right)} \\ \nonumber && \qquad \times
\frac{ ( -(b_2  + (t_2-a_2) ,  (b_1 + (t_1-a_1) )}{\left(  (b_1 + (t_1-a_1))^2 + (b_2 + (t_2-a_2) )^2 \right)}
\left( 1 - e^{{-\frac{1}{2 \lambda^2}\left(  (b_1 + (t_1-a_1))^2 + (b_2 + (t_2-a_2) )^2 \right)}}
\right)
\end{eqnarray}
If we now return to \eqref{eq:moment_evolution} we see that in order to compute the coefficients
in the moment equations we need to evaluate the divergence of this last expression
with respect to $a$ which gives
\begin{eqnarray}\label{eq:oneblobkernel}\nonumber
&& K(a_1,a_2,b_1,b_2,t_1,t_2;\lambda) = \\ \nonumber
&&  -\left( \frac{
      \left( {a_2}\,{t_1} - 
        {b_2}\,{t_1} + 
        \left( -{a_1} + {b_1} \right) \,
         {t_2} \right) }{\,\pi \,
      \left( {{a_1}}^2 + {{a_2}}^2 + 
        {{b_1}}^2 + {{b_2}}^2 + 
        2\,{b_1}\,{t_1} + 
        {{t_1}}^2 - 
        2\,{a_1}\,
         \left( {b_1} + {t_1} \right)  +
         2\,{b_2}\,{t_2} + 
        {{t_2}}^2 - 
        2\,{a_2}\,
         \left( {b_2} + {t_2} \right) 
        \right) \,\lambda^2} \right) \\ 
        && \left( -1 + 
        e^{\frac{{\left( -{a_1} + 
                 {b_1} + {t_1} \right) }^2 + {\left( -{a_2} + 
                 {b_2} + {t_2} \right) }^2}{2\,\lambda^2}} \right) e^{\frac{{{a_1}}^2 + {{a_2}}^2 - 
           2\,{a_1}\,{b_1} + 
           {{b_1}}^2 - 
           2\,{a_2}\,{b_2} + 
           {{b_2}}^2 + 
           2\,{a_1}\,{t_1} + 
           2\,{b_1}\,{t_1} + 
           {{t_1}}^2 + 
           2\,{a_2}\,{t_2} + 
           2\,{b_2}\,{t_2} + 
           {{t_2}}^2}{2\,\lambda^2}}
\end{eqnarray}
Returning to equation  \eqref{eq:moment_evolution}  we finally conclude that 
\begin{eqnarray} \label{eq:moment_evolution_redux}
 && \frac{d M}{dt}[k_1,k_2,t] = \\
 \nonumber && \quad = - \frac{ (-1)^{(k_1+k_2)} \lambda^{2 (k_1+ k_2)}}{2^{k_1+k_2} (k_1 !) (k_2 !) } \sum_{\ell_1,\ell_2} \sum_{m_1,m_2}  \Gamma[k_1,k_2;\ell_1,\ell_2,m_1,m_2;\lambda]
  M[\ell_1,\ell_2,t]  M[m_1,m_2,t] \end{eqnarray}
where
\begin{eqnarray}\label{eq_coefficient_def}
&& \Gamma[k_1,k_2;\ell_1,\ell_2,m_1,m_2;\lambda] = \\ \nonumber
 && \qquad =  D_{t_1}^{k_1} D_{t_2}^{k_2} D_{b_1}^{m_1} D_{b_2}^{m_2} D_{a_1}^{\ell_1} D_{a_2}^{\ell_2} K(a_1,a_2,b_1,b_2,t_1,t_2;\lambda)|_{t=0,a=0,b=0}
\end{eqnarray}
Thus, we have succeeded in rewriting the two-dimensional vorticity equation as
a system of ordinary differential equations with simple, quadratic nonlinear terms
whose coefficients can be evaluated in terms of derivatives of a single explicit
function.  Furthermore, we have given a sufficient condition on the initial vorticity distribution
to guarantee that the expansion of the vorticity generated by the solution of these
ordinary differential equations converges for all time.

\section{The moment expansion for several vortex centers}\label{sec:four}

In this section we extend the Hermite moment expansion of the previous section to 
the case in which there are two or more centers of vorticity by combining this
expansion with the multi-vortex representation of Section \ref{sec:two}.
For simplicity of exposition we limit the discussion here to the case of two vortices but
the expansion can be extended to any finite number of vortices.

The basic idea is just to consider the equations \eqref{eq:omegaj} for the evolution
of each vortex and then expand each of the functions $\omega^j$ in Hermite
moments as in the previous section.   Thus, we define
\begin{equation}\label{eq:moment_expansion_def_tw0}
\omega^j(z,t) = \sum_{k_1, k_2=1}^{\infty} M^j[k_1,k_2;t] \phi_{k_1,k_2}(z,t;\lambda) 
\end{equation}
for $j=1,2$.
We make a similar expansion for the velocity field in terms of the functions
$\VV_{\ell_1,\ell_2}$, and insert the expansions into \eqref{eq:omegaj}.  Letting
$z = x-x^j(t)$, and recalling that  $\partial_t \phi_{k_1,k_1} = \nu
\Delta \phi_{k_1,k_2}$  we obtain:
\begin{eqnarray}\label{eq:mom_multi}
&& \frac{dM^j [k_1,k_2;t] }{dt} = \\ \nonumber && 
\qquad = - P_{k_1,k_2} \left[  \left( \sum_{j^{\prime}=1}^2
\sum_{\ell_1, \ell_2=1}^{\infty} M^{j^{\prime}} [\ell_1,\ell_2;t] 
\VV_{\ell_1,\ell_2}(z-s_{j,j^{\prime}},t;\lambda) \right)
 \right.  \\ \nonumber &&
\qquad \qquad \qquad \qquad \left. \cdot \nabla \left(\sum_{k_1, k_2=1}^{\infty} M^j[m_1,m_2;t] \phi_{m_1,m_2}(z,t;\lambda)  \right) \right]
\end{eqnarray}
where $s_{j,j^{\prime}} = x^{j^{\prime}}(t) - x^j(t)$.
{Proceeding as in the previous section and using equation  \eqref{eq:moment_evolution}  we finally conclude that }
\begin{eqnarray} \label{eq:moment_evolution_multi}
 && \frac{d M^j}{dt}[k_1,k_2,t] = \\
 \nonumber && \quad = - \frac{ (-1)^{(k_1+k_2)} \lambda^{2 (k_1+ k_2)}}{2^{k_1+k_2} (k_1 !) (k_2 !) } \sum_{j^{\prime}=1}^2 \sum_{\ell_1,\ell_2} \sum_{m_1,m_2}  \Gamma^{j,j^{\prime}}[k_1,k_2;\ell_1,\ell_2,m_1,m_2;s_{j,j^{\prime}},\lambda] \\ \nonumber
 && \qquad \qquad  \qquad 
 \times  M^j[\ell_1,\ell_2,t]  M^{j^{\prime}}[m_1,m_2,t] 
 \end{eqnarray}
where
\begin{eqnarray}\label{eq:coefficient_def_multi}
&& \Gamma^{j,j^{\prime}}[k_1,k_2;\ell_1,\ell_2,m_1,m_2;s_{j,j^{\prime}},\lambda] = \\ \nonumber
 && \qquad =  D_{t_1}^{k_1} D_{t_2}^{k_2} D_{b_1}^{m_1} D_{b_2}^{m_2} D_{a_1}^{\ell_1} D_{a_2}^{\ell_2} K^{multi}(a_1,a_2,b_1,b_2,t_1,t_2;s_{j,j^{\prime}},\lambda)|_{t=0,a=0,b=0}
\end{eqnarray}
and 
\begin{eqnarray}\label{eq:K_multi_def}
&&  K^{multi}(a_1,a_2,b_1,b_2,t_1,t_2;s_1,s_2 ,\lambda) = \\ \nonumber
&& \qquad \qquad 
\nabla_{a} \cdot
\left(   \int_{\real^2} e^{\frac{ (-t_1^2-t_2^2 + 2 t_1 x_1 + 2t_2 x_2)}{\lambda^2}} \VV_{00}(x-s+b;\lambda) 
\phi_{00}(x+a;\lambda) d\xx \right)  
\end{eqnarray}

\begin{remark}
Note that comparing \eqref{eq:coefficient_def_multi} with
the calculation leading \eqref{eq:oneblobkernel} we see that 
$K^{multi}$ can be written in terms of the expression for $K$ via the
simple formula:
$$
K^{multi}(a_1,a_2,b_1,b_2,t_1,t_2;s_1,s_2 ,\lambda) =K(a_1,a_2,b_1-s_1,b_2-s_2,t_1,t_2;\lambda) 
$$
\end{remark}

{
We now derive a similar expansion for the evolution of the centers of each vortex.  We begin
with \eqref{eq:xjdot}
$$
\frac{d x^j}{dt}(t) = \frac{1}{m_j} \sum_{ j^{\prime}=1}^N \inttwo  \left(
\uu^{\ell} (z + x^j(t) -x^{j^{\prime}}(t),t) \omega^j(z,t)  \right) dz\ .
$$
(Recall that this is really a pair of equations, one for each component of $x^j$.)  Now insert
the moment expansion of $\uu^{j^{\prime}}$ and $\omega^j$ into this expression and we
obtain:
\begin{equation}
\frac{d x^j}{dt}(t) = \sum_{j^{\prime}=1}^N \sum_{\ell_1,\ell_2} \sum_{m_1,m_2} 
  M^j[\ell_1,\ell_2,t]  M^{j^{\prime}}[m_1,m_2,t] 
  \int_{\real} \VV_{\ell_1,\ell_2}(z-s_{j,j^{\prime}},t) \phi_{k_1,k_2}(z,t) dz\ ,
\end{equation}
where as before $s_{j,j^{\prime}}= x^{j^{\prime} }(t)-x^j(t)$.
The integral of the velocity and vorticity can be evaluated just as in the preceding section
and we find
 \begin{equation}\label{eq:xdotmoment}
 \frac{d x^j}{dt}(t) =
\sum_{j^{\prime}=1}^N \sum_{\ell_1,\ell_2} \sum_{m_1,m_2}  \Xi^{j,j^{\prime}}[\ell_1,\ell_2,m_1,m_2;s_{j,j^{\prime}},\lambda]
M^j[\ell_1,\ell_2,t]  M^{j^{\prime}}[m_1,m_2,t] \ ,
 \end{equation}
 where in this case the coefficients $\Xi^{j,j^{\prime}}[k_1,k_2;\ell_1,\ell_2,m_1,m_2;s_{j,j^{\prime}},\lambda]$ is given by the expression
 \begin{eqnarray}\label{eq:xdotcoefficient}
&& \Xi^{j,j^{\prime}}[\ell_1,\ell_2,m_1,m_2;s,\lambda]  \\ \nonumber && \qquad =
 \frac{1}{2\pi}  D_{b_1}^{m_1} D_{b_2}^{m_2} D_{a_1}^{\ell_1} D_{a_2}^{\ell_2} \left(
\big( \frac{ (-(b_2-s_2-a_2),(b_1-s_1-a-1)) }{(b_2-s_2-a_2)^2+(b_1-s_1-a-1)^2 } \big) \right. \\ \nonumber
&& \qquad \qquad 
\left. \left( 1- e^{-\frac{1}{2\lambda^2} ((b_2-s_2-a_2)^2+(b_1-s_1-a-1)^2 )} \right) \right))|_{a=0,b=0}\ .
\end{eqnarray}

}

{
\subsection{Convergence of Multi-Vortex Expansion}

We note that it is easy to extend our previous result on the
convergence of the moment expansion
 to the multi-vortex expansion.  To do so, we  change variables to re-center each 
 vortex at the origin. Thus if we let 
$$\omega^j(x-x^j(t),t) = w^j(x,t), \hspace{.5cm}\uu^j(x-x^j(t),t) = \vv^j(x,t)$$
{ then  equation {(\ref{eq:omegaj})} becomes}

\begin{eqnarray}\label{eq:omegaNew}\nonumber
\frac{\partial w^j}{\partial t}(x,t) &=& \nu \Delta w^j(x,t) 
-  \left(\sum_{\ell=1}^N 
\vv^{\ell}(x-x^{\ell}(t)+x^{j}(t),t) \right) \cdot \nabla w^j(x,t),
\end{eqnarray}
for $j=1, \dots , N$.

We are now ready to state our result.

\begin{theorem}
Define
\begin{equation*} 
\mathcal{E}^j(t) = \int_{\mathcal{R}^2} G^{-1}_\lambda(x)w^j(x,t)^2dx.
\end{equation*}
If the initial vorticity distribution $\omega_0 \equiv \sum_{\ell=1}^N \omega^\ell(x-x^\ell(0),0)$ is such that $\mathcal{E}^j(0) < \infty$ for some $\lambda_0$,
and for all $j=1,2, \dots, N$, and if 
$\omega_0$ is bounded (in the $L^\infty$ norm) then each $\mathcal{E}^j(t)$ is finite for all times $t>0$.
\end{theorem}
\begin{proof}
{
We use the same idea as the single vortex case and differentiate $\mathcal{E}^j(t)$:
\begin{eqnarray*}
\frac{d\mathcal{E}^i}{dt}(t) & = & \frac{4\nu}{\lambda^2}\mathcal{E}^i(t)
 -\frac{4\nu}{\lambda^4} \int_{\mathbb{R}^2}|x|^2 G^{-1}_\lambda(x)w^i(x)^2dx \\
&   & + 2 \int_{\mathbb{R}^2} G^{-1}_\lambda(x)w^i(x,t)(\partial_t w^i(x,t)dx\\
& = & \frac{4\nu}{\lambda^2}\mathcal{E}^i(t)
 -\frac{4\nu}{\lambda^4} \int_{\mathbb{R}^2}|x|^2 G^{-1}_\lambda(x)(w^i(x)^2dx \\
&   & + 2 \int_{\mathbb{R}^2} G^{-1}_\lambda w^i(\nu\Delta w^i-
\left(\sum_{\ell=1}^N 
\vv^{\ell}(x-x^{\ell}(t)+x^{j}(t),t) \right)\cdot\nabla w^i)  dx.
\end{eqnarray*}}
From here the proof proceeds identically to the proof of theorem \ref{prop:convergence_criterion} except we must examine the nonlinear term that comes from the vorticity equation a bit closer. In general, it is unknown whether $w^j(\cdot,t)$ satisfies a maximum principle, but since$ \left(\sum_{\ell=1}^N 
\vv^{\ell}(x-x^{\ell}(t)+x^{j}(t),t) \right)= \uu$, { the solution to (\ref{fluid1}),} 
it does satisfy a maximum principle.   Hence, as in Subsection \ref{subsec:onevortexconvergence}
we can bound the $L^{\infty}$ norm of $\sum_{\ell=1}^N 
\vv^{\ell}(x-x^{\ell}(t)+x^{j}(t),t) $ by a constant depending only on the {\em initial}
vorticity distribution.
 Hence we proceed to bound the integral
\begin{eqnarray*}
 \int_{\mathbb{R}^2} G^{-1}_\lambda 
 w^i( \left(\sum_{\ell=1}^N 
\vv^{\ell}(x-x^{\ell}(t)+x^{j}(t),t) \right)\cdot\nabla w^i)  dx \le 2C(\omega_0) \int_{\mathbb{R}^2} 
G^{-1}_\lambda |w^i| |\nabla w^i| dx
\end{eqnarray*}
and thus the rest of the bounds are the same. Again, putting everything together we arrive at:
\begin{equation}
\frac{d\mathcal{E}^i}{dt}(t) \le \left(\frac{4\nu}{\lambda^2}+\frac{4C(\omega_0)}{\nu}\right)
\mathcal{E}^i(t) 
\end{equation}
\end{proof}

}

\subsection{Interaction of Gaussian Vortices}\label{sec:Kirchoff_generalization}

The experimental and numerical work of \cite{Meunier:2005} has shown that widely separated
regions of vorticity can be well approximated by { Gaussians} for long periods of time.
This corresponds to truncating our expansions for the vorticity so that they
contain only a single term.
In this subsection we
analyze the equations that result from this truncation.  {This approximation can be viewed as a generalization of the Helmholtz-Kirchhoff approximation in which}
we include the effects of vorticity and finite core size to lowest order.  We show that the
total vorticity of each of the vortices is constant while the centers of vorticity
evolve either along straight lines or circles.  

{
\begin{remark} It should be noted here that the effect of truncating our expansion after one term allows for only viscosity as the driving force in vortex merger. Allowing more terms in the expansion introduces convective forces.
\end{remark}
}

If we start with two {Gaussian distributions for our initial} vorticity with the same value
of $\lambda_0$ , and
truncate the equations of motion for the moments so that all terms containing higher order
moments are omitted then we can conclude that the two vortices  travel along circular or straight line orbits around the ``center of vorticity."  Moreover, the leading coefficients $M^j[0,0:t]$ of the expansions are constant.  To be precise, we let
\begin{eqnarray}
\omega^1(x,t) &=&  M^1[0,0;t]\phi_{0,0}(x-x^1(t),t;\lambda)\\
\omega^2(x,t) &=& M^2[0,0;t]\phi_{0,0}(x-x^2(t),t;\lambda).
\end{eqnarray}
Let us also write $s_{i,j}=x^j(t)-x^i(t)$, then using { (\ref{eq:mom_multi})- (\ref{eq:K_multi_def})} we first  calculate the evolution of $M^1(0,0:t)$


 { 
\begin{eqnarray*}
\frac{dM^1}{dt}[0,0;t] & = & -M^1[0,0;t]^2K(0,0,0,0,0,0,\lambda) \\ &  &  - M^1[0,0;t]M^2[0,0;t]K^{multi}(0,0,0,0,0,0,s_{1,2},\lambda)  \\ & = & 0,
\end{eqnarray*}
since $K(0,0,0,0,0,0,\lambda)=0$ and $K^{multi}(0,0,0,0,0,0,\lambda)=0$, respectively. The calculation for $M^2[0,0;t]$ is the same.
}
Thus the evolutions of the coefficients of leading order are constant. 

\begin{remark} The fact that $M^1[0,0;t]$ is constant in time is not a consequence of the truncation
of the moment equations to first order.  One can show that the equations \eqref{eq:allomega}
conserve the zeroth moment of $\omega^j$, independent of any truncation.
\end{remark}

More interestingly, though, is the calculation for the evolution of $x^j(t)$, the centers of these vortices. Again if we denote $x^j = (x^j_1,x^j_2)$ then the evolution for each component as defined by equation (\ref{eq:xjdot}) can be written as:
\begin{equation}\label{vfld2vrtx}
\frac{dx^j_i}{dt}  = \frac{1}{M^j[0,0,t]} \int V_i^k(y-s_{1,2},t)\omega^j(y,t)dy.
\end{equation}
Now using equation {(\ref{eq:stream})} to evaluate
$\\V^k(j-s_{1,2},t)$, equation (\ref{vfld2vrtx}) yields the following equations for 
each $x^j_i$, $j=1,2$ and $i=1,2$:

\begin{equation}\label{FullSys}
\begin{array}{ccc}
\dot{x}^1_1 & = & -\frac{M^2}{2\pi} \frac{(e^{-\frac{(x^1_1 - x^2_1)^2+(x^1_2 - x^2_2)^2}{2\lambda(t)^2}}-1)(x^1_2 - x^2_2)}{(x^1_1 - x^2_1)^2+(x^1_2 - x^2_2)^2} \vspace{.25cm}
\\ \vspace{.25cm}
\dot{x}^1_2 & = & -\frac{M^2}{2\pi} \frac{(e^{-\frac{(x^1_1 - x^2_1)^2+(x^1_2 - x^2_2)^2}{2\lambda(t)^2}}-1)(-x^1_1 + x^2_1)}{(x^1_1 - x^2_1)^2+(x^1_2 - x^2_2)^2}
\\ \vspace{.25cm}
\dot{x}^2_1 & = & \frac{M^1}{2\pi} \frac{(e^{-\frac{(x^1_1 - x^2_1)^2+(x^1_2 - x^2_2)^2}{2\lambda(t)^2}}-1)(x^1_2 - x^2_2)}{(x^1_1 - x^2_1)^2+(x^1_2 - x^2_2)^2}
\\ \vspace{.25cm}
\dot{x}^2_2 & = & \frac{M^1}{2\pi} \frac{(e^{-\frac{(x^1_1 - x^2_1)^2+(x^1_2 - x^2_2)^2}{2\lambda(t)^2}}-1)(-x^1_1 + x^2_1)}{(x^1_1 - x^2_1)^2+(x^1_2 - x^2_2)^2}
\end{array}
\end{equation}
where $M^{j} \equiv M^j[0,0,t]$ is constant and represents the total vorticity of the $j^{th}$ vortex.
\begin{remark}
If the vortices have different $\lambda(t)$ values say $\lambda_1(t)$ and $\lambda_2(t)$ then one just needs to replace the $2\lambda(t)^2$ with $\lambda_1(t)^2+\lambda_2(t)^2$ in the exponential to arrive at the correct system.
\end{remark}

We {now state} the main result of this section.

\begin{theorem}
System (\ref{FullSys}) admits only circular or straight line trajectories.
\end{theorem}
{\bf Proof:}
Away from rest points our system (\ref{FullSys}) can be transformed to:
{
\begin{equation}\label{FullSys2}
\begin{array}{ccc}
\frac{\partial x^1_2}{\partial x^1_1} & = & \frac{-x^1_1 + x^2_1}{x^1_2-x^2_2} \\
\frac{\partial x^2_2}{\partial x^2_1} & = & \frac{-x^1_1 + x^2_1}{x^1_2-x^2_2} \\
\frac{\partial x^1_2}{\partial x^2_2} & = & -\frac{M^2}{M^1}
\\
\frac{\partial x^1_1}{\partial x^2_1} & = & -\frac{M^2}{M^1}.
\end{array}
\end{equation}
}
Integrating out the bottom two equations { we get}
\begin{equation}
\begin{array}{ccc}
x^1_2 & = & -\frac{M^2}{M^1}(x^2_2 + k_2) \\
x^1_1 & = & -\frac{M^2}{M^1}(x^2_1 + k_1).
\end{array}
\end{equation}

First let us assume that $M^2 \ne -M^1$. Then plugging back into the first two equations of \eqref{FullSys2} we arrive at:
\begin{equation}
\begin{array}{ccccc}
\frac{\partial x^1_2}{\partial x^1_1} & = & \frac{(-1 -\frac{M^2}{M^1})x^1_1 -k_1}{(1 +\frac{M^2}{M^1})x^1_2 -k_2} & = & -\frac{x^1_1 -\tilde{k}_1}{x^1_2 -\tilde{k}_2} \\
\frac{\partial x^2_2}{\partial x^2_1} & = & \frac{(1 +\frac{M^2}{M^1})x^2_1 -\frac{M^2}{M^1}k_1}{(-1 -\frac{M^2}{M^1})x^1_2 -\frac{M^2}{M^1}k_2} & = & -\frac{x^1_1 -\hat{k}_1}{x^1_2 -\hat{k}_2}
\end{array}
\end{equation}
for appropriate constants $\tilde{k}_1,\tilde{k}_2,\hat{k}_1,\hat{k}_2$. Thus we integrate again and get
\begin{equation}
\begin{array}{ccc}
(x^1_2(t)-\tilde{k}_2)^2 + (x^1_1(t)-\tilde{k}_1)^2 = C_1 \\
(x^2_2(t)-\hat{k}_2)^2 + (x^2_1(t)-\hat{k}_1)^2 = C_2 .
\end{array}
\end{equation}

If $M^2 = -M^1$ then we have equal but opposite size vortices and equations (\ref{FullSys}) become:
\begin{equation}
\begin{array}{ccc}
\frac{\partial x^1_2}{\partial x^1_1} & = & -\frac{k2}{k1} \\
\frac{\partial x^2_2}{\partial x^2_1} & = & -\frac{k2}{k1} 
\end{array}
\end{equation}
which gives us straight line solutions with slope $-\frac{k2}{k1}$, as desired. 

\QED

{
If we now consider the case of equal total vorticity for the two vortices ($M^2=M^1\equiv M$), the classical point vortex result is that the vortices will rotate around the center of vorticity at a constant frequency, $\frac{2M}{2\pi D^2}$, where $D$ is the distance between the vortex centers. 
We will now compute the viscous and finite core size effects  on the frequency of rotation,
predicted by our model.  For simplicity we will center the vortices at the origin and place them on the circle of radius, $r$. We apply the polar change of variables,
\begin{equation}
x^i_1 = r\cos(\theta_i) \hspace{.5cm} x^i_2=r\sin(\theta_i) \hspace{.5cm} i=1,2
\end{equation}
then  using the fact that our vortices are out of phase by $\pi$ we can compute that
\begin{eqnarray}
(x^1_1-x^2_1) & = & r\cos(\theta_1) - r\cos(\theta_2) = 2r\cos(\theta_1) \\
(x^1_2-x^2_2) & = & r\sin(\theta_1) - r\sin(\theta_2) = 2r\sin(\theta_1),
\end{eqnarray}
and thus we arrive at the expression for the frequency of rotation
\begin{equation}\label{NewFreq}
\Omega = \frac{M}{4 \pi r^2} (1-e^{\frac{-2r^2}{\lambda(t)^2}}).
\end{equation}

Since $\lambda(t)^2 = \lambda_0^2 +4\nu t$ we notice two things: the first is that viscosity, $\nu$, slows the frequency of rotation down and, second, in the formal limit as $\nu \rightarrow 0$ and $\lambda_0^2 \rightarrow 0$ we recover the constant frequency, $\frac{2M}{2\pi D^2}$, which is what the Helmholtz-Kirchoff
model for  the rotation of two point vortices predicts.}

{With this calculation we may now compare the results of expanding a two gaussian  initial distribution as a single vortex or as two independent vortices.  In \cite{nagem:2007} a two Gaussian  initial distribution a distance $2r$ apart with core size $\lambda_0$, each with mass $M$, is approximated by a single vortex expansion using  (\ref{eq:moment_expansion_def}). The authors truncate the expansion to quadrapole moment ($n=2$) and calculated the frequency of rotation to be:
\begin{equation}\label{OldFreq}
\Omega = \frac{M}{8 \pi} \left[ 
\frac{1}{2 \nu t} 
\ln \left( 1 + \frac{4 \nu t}{r^{2}} \right) -
\frac{ \lambda_0^{2}}{r^{4}} \frac{1}{1 + 4 \nu t / r^{2}}
\right].
\end{equation}
This equation is directly comparable to (\ref{NewFreq}).  Notice that both equations  for frequency of rotation indicate  slowing of rotation over time, albeit at different rates.  In addition, both equations recover 
the Helmholtz-Kirchhoff approximation of $\frac{2M}{2\pi D^2}$  in the limit as $\nu \rightarrow 0$ and $\lambda_0 \rightarrow 0$. In fact, in the sufficiently localized regime, $r << \lambda_0$ both equations have similar initial frequencies and remain close asymptotically.  A typical example is shown in figure \ref{frequency}.

\begin{figure}\begin{center}\label{frequency}
``Typical" frequency of rotation for two localized vortices\\
\includegraphics[width=.4\textwidth]{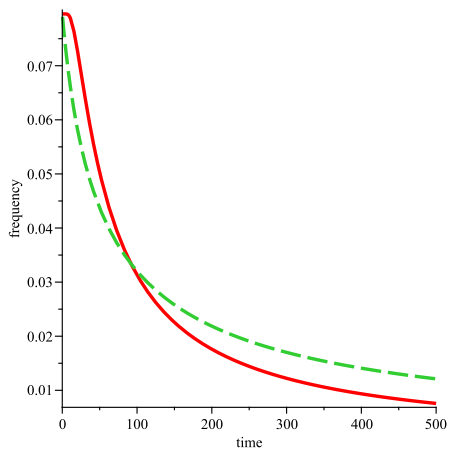}\end{center}
\caption{Here we plot both the frequency of rotation predicted by single vortex expansion up to quadrapole order (dashed) and the frequency predicted by two vortex expansions truncated to leading order (solid). The parameter values used are $\nu=.01$,$M=1$,$r=1$, and $\lambda_0 =.01$.}
\end{figure}
}

\section{Conclusions:}  In this paper we have derived a system of ordinary differential equations
whose solutions give a representation of solutions of the two dimensional vorticity equation
in terms of a system of interacting vortices.  We have also derived a sufficient condition on the
initial vorticity distribution which guarantees that this representation in terms of interacting
vortices is equivalent to the original solution of the two-dimensional vorticity equation.
This model generalizes the classical Helholtz-Kirchhoff model of interacting, inviscid,
point vortices to include the effects of both finite core size and viscosity.  We have
also looked at the analytical predictions of our model for the interaction of two vortices
which the expansion is truncated at leading order.  We plan in future work to further
explore the analytical and numerical predictions of this model.

\section{Acknowledgements:}  CEW wishes to acknowledge many useful discussions about two
dimensional fluid flows with Th.~Gallay.  DU and CEW wish also to acknowledge a very
helpful discussion of Hermite expansions with G. Van Baalen.

\bibliographystyle{plain} \bibliography{multi_reference}

\end{document}